\documentclass[11pt,oneside,english]{amsart}
\usepackage{subfiles, mathtools}
\usepackage{amssymb}
\usepackage[colorlinks]{hyperref}
\usepackage{graphicx}
\usepackage[pagewise]{lineno}
\usepackage{ulem} 

\makeatletter \theoremstyle{plain}
 \newtheorem{thm}{Theorem}[section]
 \newtheorem{lem}[thm]{Lemma}
 \newtheorem{prop}[thm]{Proposition}
 
 \numberwithin{equation}{section} 
 \numberwithin{figure}{section} 
 \theoremstyle{plain}
 \theoremstyle{definition}
 \newtheorem{defn}[thm]{Definition}
 \newtheorem{rem}[thm]{Remark}
  \newtheorem{cor}[thm]{Corollary}
  \newtheorem{ex}[thm]{Example}

\newcommand{\calC}{{{\mathcal C}}}

\newcommand{\fH}{{{\mathfrak H}}}

\newcommand{\bH}{{{\bf H}}}

\newcommand{\Aa}{{\mathcal {AA}}}

\newcommand{\calH}{{{\mathcal H}}}

\newcommand{\calQ}{{{\mathcal Q}}}

\newcommand{\C}{{{\mathbb C}}}

\newcommand{\R}{{{\mathbb R}}}

\newcommand{\mH}{{{\mathbb H}}}

\usepackage{babel}

\makeatother

\providecommand{\keywords}[1]
{
  \small	
  {\textit{Key words and phrases.}} #1
}

\providecommand{\subjclass}[1]
{
  \small	
  {\textit{$2020$ Mathematics Subject Classifications.}} #1
}

\begin{document}

\title [TW curvatures in 3D Lie groups]{Tanaka-Webster curvatures of surfaces in 3D Lie groups with a CR structure}

\author{Ioannis D. Platis \& Dimitrios Tsolis}
\thanks{The first author was supported by the Medicus program, Grant Nr. 83765.}

\newcommand{\Addresses}{{
  \bigskip
  \footnotesize

  I.D.~Platis, \textsc{Department of Mathematics, University of Patras, Panepistimioupolis, 26504 Rion, Achaia, Greece}\par\nopagebreak
  \textit{E-mail address}, I.D.~Platis: \texttt{idplatis@upatras.gr}

  \medskip
  
  D.~Tsolis, \textsc{Department of Mathematics, University of Patras, Panepistimioupolis, 26504 Rion, Achaia, Greece}\par\nopagebreak
  \textit{E-mail address}, D.~Tsolis: \texttt{d.tsolis@upatras.gr}

}}

\date{\today}

\begin{abstract}
We consider the Tanaka-Webster geometry of surfaces embedded in a 3-dimensional Lie group with a CR structure inherited by a contact form. We define the notions of Gauss and mean curvature and give specific examples.
\end{abstract}

\maketitle

\centerline{\it Dedicated to Sorin Dragomir.
}


\section{Introduction}

The study of surfaces (and submanifolds in a more general way) within the Cauchy-Riemann (CR) manifold area elegantly bridges complex analysis, contact geometry, and sub-Riemannian geometry; for standard references, see~\cite{DT-CR} and~\cite{Bejancu1986}. In this paper, we focus on a specific 3-dimensional ambient space: a Lie group $G$ equipped with a left-invariant contact form $\theta$ and its associated CR structure. To investigate the differential geometry of this space, we rely on the Webster metric and the Tanaka-Webster connection~\cite{Tanaka1975}. This connection is the unique linear connection that naturally preserves both the metric and the contact form, while maintaining the Reeb vector field and the complex structure parallel.  

Our primary goal is to explore the differential geometry of smooth hypersurfaces (surfaces) immersed in these Lie groups. By utilising the horizontal and Webster gradients, we establish a specialised, adapted orthonormal moving frame at the non-characteristic points of the surface. This framework allows us to directly project the Tanaka-Webster connection onto the surface and derive its structural equations. From this foundation, we construct the Webster second fundamental form, which is notably non-symmetric, and use it to formally define the Tanaka-Webster (TW) mean curvature and the TW Gauss curvature of the surface.  

The paper is organised as follows: In Section \ref{sec:prel}, we present the necessary geometric background on contact 3-manifolds, the Webster metric, and the structure of contact 3-dimensional Lie groups. Section \ref{sec-TW-hyp} details the TW-geometry of surfaces in $G$, constructing moving frames, deriving structural equations, and formalising the Gaussian and mean curvatures. In Section \ref{sec-comp-H}, we apply these tools to the Heisenberg group $\mathbb{H}$, performing explicit computations for graphs, tilted planes, and surfaces of revolution. Section \ref{sec-comp-aa} offers similar, detailed computations for surfaces embedded in the affine additive group $\Aa$. Finally, in Section \ref{sec-sasaki}, we briefly discuss the Sasakian structure in G and examine the behaviour of curvatures under Riemannian approximations.

\section{Preliminaries}\label{sec:prel}
In this section, we review the fundamental concepts and geometric structures that underpin our analysis. In Section \ref{sec:contact}, we recall the definitions of contact 3-manifolds, the associated CR-structure, the Webster metric, and the properties of the Tanaka-Webster connection. In Section \ref{sec:groups}, we specialise to contact 3-dimensional Lie groups and establish the existence of a canonical left-invariant orthonormal frame. Finally, in Sections \ref{sec:tw-group} and \ref{sec:curv}, we derive the structural equations for the Tanaka-Webster connection in this Lie group setting and compute the corresponding Webster scalar curvature.

\subsection{Contact 3-manifolds and Webster metric}\label{sec:contact}
 Let $M$ be a 3-dimensional smooth manifold with a codimension 1 contact form $\theta$, that is a smooth 1-form satisfying
$$
\theta\wedge d\theta\neq 0,
$$
everywhere in $M$. Due to continuity, we may assume that $\theta\wedge d\theta>0$. The form $\theta$ defines a 2-dimensional horizontal distribution $\mathcal{H}=\ker\theta$ in $M$. Suppose that $\calH$ is spanned by vector fields $X,Y$. The CR-structure $J:\mathcal{H}\to\mathcal{H}$ in $M$ may be defined by setting
$$
JX=Y,\quad JY=-X.
$$
Associated with the contact structure is the Reeb vector field $T$, which is the unique vector field satisfying
$$
\theta(T)=1,\quad d\theta(T,\cdot)=0.
$$
We can extend $J$ to the entire tangent bundle $TM$ by setting $JT=0$. Recall that for a horizontal vector field $V$ we have
$$
0<(\theta\wedge d\theta)(T,V,JV)=\theta(T)d\theta(V,JV)+\theta(V)d\theta(JV,T)+
\theta(JV)d\theta(T,V)=d\theta(V,JV).
$$
\begin{defn}
The {\it Webster metric} $g$ in $M$ is defined by:
\begin{equation}\label{eq:W-metric}
g(X_1,X_2)=d\theta(X_1,JX_2)+\theta(X_1)\theta(X_2),
\end{equation}
for all $X_1,X_2\in TM$.
\end{defn}
For the Webster metric in the broader context of pseudo-Hermitian geometry, see also~\cite{Lee1988}.
\subsubsection{Tanaka-Webster connection}
The {\it Tanaka-Webster connection} $\nabla$ (see \cite{DT-CR}) is the unique linear connection in $M$ that satisfies the following conditions:
\begin{enumerate}
    \item[{1)}] The connection preserves the Webster metric and the contact form, that is
    $$
    \nabla g=0,\quad \nabla\theta=0.
    $$
    \item[{2)}] The Reeb vector field $T$ and the complex structure $J$ are parallel, that is,
    $$
    \nabla T=0,\quad \nabla J=0.
    $$
    \item [{3)}] The torsion tensor ${\rm Tor}$ defined by
    $$
    {\rm Tor}(X_1,X_2)=\nabla_{X_1}X_2-\nabla_{X_2}X_1-[X_1,X_2]
    $$
    satisfies the following:
    \begin{eqnarray*}
    &&
    {\rm Tor}(X_1,X_2)=d\theta(X_1,X_2)T,\quad X_1,X_2\in\mathcal{H},\\
     &&
    {\rm Tor}(T, J X_1)=-J{\rm Tor}(T,X_1),\quad X_1\in\mathcal{H}.
    \end{eqnarray*}
    \item[{4)}]  For $X_1\in\mathcal{H}$, ${\rm Tor}(T,X_1)=A(X_1)$,
    where $A$ is an anti-linear operator (the torsion operator) satisfying $AJ=-JA$.
\end{enumerate}
\subsection{Contact 3-Lie groups}\label{sec:groups}
Let $G$ be a 3-dimensional Lie group with a contact form $\theta$, satisfying
$
\theta\wedge d\theta> 0,
$
everywhere in $G$ and let $T$ be the Reeb vector field. Let $\mathcal{H}=\ker\theta$ and let $X,Y$ be left-invariant vector fields in $G$ such that $\mathcal{H}=\langle X,Y\rangle$. Again, the CR-structure $J:\mathcal{H}\to\mathcal{H}$ in $G$ is defined by setting
$
JX=Y,\quad JY=-X.
$
\begin{prop}
  Let $G$ and $X,Y, T$ be as above. Then there exist $e_1,e_2\in\mathcal{H}$ such that $\{e_1,e_2,T\}$ is an orthonormal frame for the Webster metric, $Je_1=e_2$, $Je_2=-e_1$, $JT=0$ and moreover, there exist $a_i,b_i\in\R$, $i=1,2,3$ such that  
 \begin{equation}\label{eq-a,b}
[e_1,T]=a_1\,e_1+b_1\,e_2, \quad [e_2,T]=a_2\,e_1+b_2\,e_2, \quad [e_1,e_2]=a_3\,e_1+b_3\,e_2-T.
\end{equation}
Moreover, $a_i,b_i$, $i\in\{1,2,3\}$ satisfy
\begin{equation}\label{eq:liebr-const}
   \left|\begin{matrix}
    a_1&b_1\\
    a_3&b_3
\end{matrix}\right|=
\left|\begin{matrix}
    a_2&b_2\\
    a_3&b_3
\end{matrix}\right|=0,\quad a_1+b_2=0.
\end{equation}
\end{prop}
\begin{proof}
  Without the assumption $\theta\wedge d\theta>0$, it is proved in~\cite{BPPT} that there exist constants $a_i,b_i$, $i\in\{1,2,3\}$ and $c< 0$ such that
$$
[X,T]=a_1\,X+b_1\,Y, \quad [Y,T]=a_2\,X+b_2\,Y, \quad [X,Y]=a_3\,X+b_3\,Y+c\,T,
$$ 
and moreover, $a_i,b_i$, $i\in\{1,2,3\}$ and $c$ satisfy Eqs. \eqref{eq:liebr-const}. Now set
    $$
    e_1=(-c)^{-1/2}X,\quad e_2=(-c)^{-1/2}Y.
    $$
    Then the frame $\{e_1,e_2,T\}$ is orthonormal with respect to the Webster metric in $G$. In fact, we have
 $$
d\theta(X,Y)=X\theta(Y)-Y\theta(X)-\theta([X,Y])=-c.
 $$
 Therefore,
 $$
 g(X,X)=d\theta(X,Y)+\theta(X)\theta(X)=-c,
 $$
 and in the same manner, we also have $g(Y,Y)=-c.$ Moreover,
 $$
 g(T,T)=\theta(T)\theta(T)=1,
 $$
 and
 \begin{eqnarray*}
     g(X,Y)&=&d\theta(X,-X)+\theta(X)\theta(Y)=0,\\
     g(X,T)&=&d\theta(X,0)+\theta(X)\theta(T)=0,\\
     g(Y,T)&=&d\theta(Y,0)+\theta(Y)\theta(T)=0.
 \end{eqnarray*}
 Now,
$$
[e_1,T]=a_1e_1+b_1e_2,\quad [e_2,T]=a_2e_1+b_2e_2,\quad [e_1,e_2]=\frac{a_3}{\sqrt{-c}}e_1+
\frac{b_3}{\sqrt{-c}}e_2-T,
$$
where $a_i,b_i$, $i=1,2,3$ and $c$ satisfy \eqref{eq:liebr-const}. We can therefore always assume that $c=-1$ under the assumption $\theta\wedge d\theta>0$ and stick with Eqs. \eqref{eq-a,b}. This concludes the proof.
\end{proof}

\subsection{Structural equations for the Tanaka-Webster connection}\label{sec:tw-group}

We consider the orthonormal frame  $\{e_1,e_2,T\}$ and let $\{\theta_1,\theta_2,\theta\}$ be the dual coframe. We will use the Lie brackets to calculate $d\theta_1,d\theta_2$ and $d\theta$. In the first place,
\begin{eqnarray*}
    d\theta_1(e_1,e_2)&=&-\theta_1([e_1,e_2])=-a_3,\\
    d\theta_1(e_i,T)&=&-\theta_1([e_i,T])=-a_i,\quad i=1,2.
\end{eqnarray*}
Similarly,
\begin{eqnarray*}
    d\theta_2(e_1,e_2)&=&-\theta_2([e_1,e_2])=-b_3,\\
    d\theta_2(e_i,T)&=&-\theta_2([e_i,T])=-b_i,\quad i=1,2,
\end{eqnarray*}
and
\begin{eqnarray*}
    d\theta(e_1,e_2)&=&-\theta([e_1,e_2])=1,\\
    d\theta(e_i,T)&=&-\theta([e_i,T])=0,\quad i=1,2.
\end{eqnarray*}
Therefore,
\begin{eqnarray}
  \notag d\theta_1&=&-a_3\,\theta_1\wedge\theta_2-a_1\,\theta_1\wedge\theta-a_2\,\theta_2\wedge\theta,\\
\label{eq:dth}     d\theta_2&=&-b_3\,\theta_1\wedge\theta_2-b_1\,\theta_1\wedge\theta-b_2\,\theta_2\wedge\theta,\\
\notag d\theta&=&\theta_1\wedge \theta_2.
\end{eqnarray}

\begin{prop}
  (See \cite{E-N}) There exist a connection form $\omega_{12}$ and torsion forms $\tau_i$, $i=1,2$ such that 
  \begin{equation}\label{firs-str}
\left(\begin{matrix}
d\theta_1\\
d\theta_2\\
d\theta
\end{matrix}\right)+
\left(\begin{matrix}
0&\omega_{12}&0\\
\omega_{21}&0&0\\
0&0&0
\end{matrix}\right)\wedge
\left(\begin{matrix}
\theta_1\\
\theta_2\\
\theta
\end{matrix}\right)=
\left(\begin{matrix}
\tau_1\\
\tau_2\\
\theta_1\wedge\theta_2
\end{matrix}\right),\quad -\omega_{12}=\omega=\omega_{21}.
\end{equation}
\end{prop}
\begin{proof}
Let
\begin{equation}\label{eq:omega}
\omega=-a_3\,\theta_1-b_3\,\theta_2-\frac{b_1-a_2}{2}\,\theta,
\end{equation}
and
\begin{equation}\label{eq-torsionforms}
\tau_1=\frac{a_1-b_2}{2}\,\theta\wedge \theta_1+\frac{a_2+b_1}{2}\,\theta\wedge\theta_2,\quad
\tau_2=\frac{a_2+b_1}{2}\,\theta\wedge\theta_1-\frac{a_1-b_2}{2}\,\theta\wedge\theta_2.
\end{equation}
Using Eqs. \eqref{eq:dth}, one verifies straightforwardly that Eqs. \eqref{firs-str} hold.
\end{proof}
\begin{defn}
Eqs. \eqref{firs-str} are the {\it first structural equations} for the Tanaka-Webster connection in $(G,g)$. The connection form is $\omega$ and the torsion forms are $\tau_i$, $i=1,2$.
\end{defn}
For an arbitrary vector field $U$, the covariant derivatives of $e_j$ with respect to $U$ are given by
$$
\nabla_Ue_j=\sum_{i=1}^2\omega_{ij}(U)e_i, \quad j=1,2,
$$
and $\nabla_UT=0$. Therefore,
\begin{eqnarray*}
&&
\nabla_Ue_1=-\omega_{12}(U)\,e_2=-\left(a_3\,\theta_1(U)+b_3\,\theta_2(U)+\frac{(b_1-a_2)\,\theta(U)}{2}\right)\,e_2,\\
&&
\nabla_Ue_2=\omega_{12}(U)\,e_1=\left(a_3\,\theta_1(U)+b_3\,\theta_2(U)+\frac{(b_1-a_2)\,\theta(U)}{2}\right)\,e_1.
\end{eqnarray*}
Explicitly,
\begin{eqnarray}
   \notag &&
    \nabla_{e_1}e_1=-a_3\,e_2,\quad \nabla_{e_1}e_2=a_3\,e_1,\quad \nabla_Te_1=((a_2-b_1)/2)\,e_2,\\
    \label{conn-forms-e}&&
    \nabla_{e_2}e_1=-b_3\,e_2,\quad \nabla_{e_2}e_2=b_3\,e_1,\quad \nabla_Te_2=((b_1-a_2)/2)\,e_1,\\
    \notag&&
    \nabla_{e_1}T=\nabla_{e_2}T=\nabla_TT=0.
\end{eqnarray}
We also deduce the following:
\begin{eqnarray*}
&&
{\rm Tor}(e_1,e_2)=T=d\theta(e_1,e_2)\,T,\\
&&
{\rm Tor}(e_1,T)=-C\,e_1-D\,e_2=\tau_1(e_1,T)\,e_1+\tau_2(e_1,T)\, e_2,\\
&&
 {\rm Tor}(e_2,T)=-D\,e_1+C\,e_2=
 \tau_2(e_1,T)\,e_1+\tau_2(e_2,T)\,e_2.
\end{eqnarray*}
Finally, the matrix of the antilinear operator $A$ with respect to the basis $\{e_1,e_2\}$ of $\calH$ is 
$$
A=\left(\begin{matrix}
C&D\\
D&-C
\end{matrix}\right),\quad C=\frac{a_1-b_2}{2}, \,\,D=\frac{a_2+b_1}{2}.
$$
\subsection{Curvature}\label{sec:curv}
There is only one curvature 2-form  $\Omega_{12}$ which is defined by the relation
$$
\Omega_{12}=d\omega_{12}-\omega_{12}\wedge\omega_{ 12}=d\omega_{12},
$$
where the latter holds due to the fact that $\omega_{12}$ is a 1-form. Straightforward calculations induce the following.
\begin{lem}
$$
\Omega_{12}=-d\omega=-\left(a_3^2+b_3^2+\frac{a_2-b_1}{2}\right)\,\theta_1\wedge\theta_2-(a_1a_3+b_1b_3)\,\theta_1\wedge \theta-(a_2a_3+b_2b_3)\,\theta_2\wedge \theta.
$$
\end{lem}
\begin{prop}
    The Webster scalar curvature $R$ of $G$ is the sectional curvature of the characteristic plane spanned by $e_1,e_2$:
    $$
    R=\Omega_{12}(e_1,e_2)=-a_3^2-b_3^2-\frac{a_2-b_1}{2}.
    $$
\end{prop}
For the affine additive group $\Aa$, we consider the contact form 
    \section{TW-geometry of hypersurfaces in $G$}\label{sec-TW-hyp}
    For a general discussion of hypersurfaces in CR geometry, see~\cite{ChernMoser1974} fand or profound studies on surfaces in the Heisenberg and broader Carnot groups, we refer the reader to~\cite{CHMY2005,DGN2001}. Here, we set up the notation in Section \ref{sec-prel-surf} and we define the adapted frame and the adapted coframe in Section \ref{sec-ort-fr}. The structural equations are given in Section \ref{sec-structural} and the second fundamental form and the mean curvature are in Section \ref{sec-sec-form}. Finally, the Gauss formula and the Gauss curvature are in Section \ref{sec-Gauss}.
    \subsection{Preliminaries and notation}\label{sec-prel-surf}
    Let $u:G\to \R$ be a smooth function and let $\Sigma$ be the hypersurface defined by the equation $u=0$.
\begin{defn}
  Let $\Sigma$ be as above. The {\it unit normal} $n_\Sigma$ of $\Sigma$ is given by:
    $$
    n_\Sigma=\frac{\nabla_Wu}{\|\nabla_Wu\|}.
    $$
    Here,
    $$
    \nabla_Wu=e_1u\,e_1+e_2u\,e_2+Tu\,T
    $$
    is the {\it Webster gradient} of $u$ and 
    $$
    \|\nabla_Wu\|=\sqrt{(e_1u)^2+(e_2u)^2+(Tu)^2},
    $$
    is the {\it Webster norm} of the Webster gradient of $u$. The {\it horizontal gradient} of $u$ is
    $$
    \nabla_Hu=e_1u\,e_1+e_2u\,e_2
    $$ 
    and the {\it horizontal norm} of $u$ is
    $$
    \|\nabla_Hu\|=\sqrt{(e_1u)^2+(e_2u)^2}.
    $$
A point $p\in\Sigma$ is called {\it characteristic} if $\nabla_Hu(p)=(0,0)$. The set of characteristic points of $\Sigma$ is the {\it characteristic locus} $\calC(\Sigma)$ of $\Sigma$. At non-characteristic points we have the {\it horizontal normal} 
$$
n_\Sigma^H=\frac{\nabla_Hu}{\|\nabla_Hu\|}.
$$
\end{defn}
\begin{rem}
It is important to emphasise the geometric distinction between the local and global geometry of $\Sigma$ relative to the horizontal distribution $\mathcal{H}$. At a non-characteristic point $p \in \Sigma \setminus \mathcal{C}(\Sigma)$, the tangent plane of the surface $T_p\Sigma$ is strictly transverse to the horizontal contact plane $\mathcal{H}_p$. At such points, the two planes intersect along a one-dimensional line, which uniquely determines the horizontal normal $n_\Sigma^H$ and allows the smooth construction of the orthonormal adapted moving frame. 

In contrast, at a characteristic point $p \in \mathcal{C}(\Sigma)$, the surface's tangent plane perfectly aligns with the ambient horizontal distribution such that $T_p\Sigma = \mathcal{H}_p$. At these singular points, transversality fails, the horizontal gradient vanishes, and the geometric framing collapses because the surface normal and the horizontal normal become linearly dependent. Globally, any closed surface embedded in a contact $3$-manifold must possess a non-empty characteristic locus due to topological constraints. Consequently, the structural equations and the resulting Tanaka-Webster curvatures derived herein are fundamentally local, capturing the behaviour of the surface exclusively on the open sets where it remains transverse to the contact structure.
\end{rem}
From this point on, we shall adopt the following notation:
\begin{eqnarray*}
&&
    l_W=\|\nabla_Wu\|,\quad l=\|\nabla_Hu\|,\\
    &&
    \overline{p}=e_1u/l,\quad \overline{q}=e_2u/l,\quad \overline{r}=Tu/l,\\
    &&
    \overline{p_W}=e_1u/l_W,\quad \overline{q_W}=e_2u/l_W,\quad \overline{r_W}=Tu/l_W.
\end{eqnarray*}
Therefore,
\begin{equation*}\label{eq-surf-W-norm}
n_\Sigma=\overline{p_W}\,e_1+\overline{q_W}\,e_2+\overline{r_W}\,T,
\end{equation*}
so that at non-characteristic points we may equivalently write
$$
n_\Sigma=(l/l_W)(\overline{p}\,e_1+\overline{q}\, e_2+\overline{r}\,T).
$$
Also, we have the {\it unit horizontal normal}
\begin{equation*}\label{eq-surf-H-norm}
n_\Sigma^H=\overline{p}\,e_1+\overline{q}\,e_2,
\end{equation*}
which is non-zero only at non-characteristic points.
\subsection{Orthonormal frame/coframe}\label{sec-ort-fr}
The tangent bundle ${\rm T}G$ splits at every non-characteristic point $p\in\Sigma$ as follows:
$$
T_p(G)=T_p(\Sigma)\oplus{\rm span}(n_\Sigma).
$$
This is shown by considering the differential
$$
du=e_1u\theta_1+e_2u\theta_2+Tu\,\theta.
$$
Then, the linearly independent unit vector fields
\begin{eqnarray}
    \label{E1} E_1&=&-\overline{q}\,e_1+\overline{p}\,e_2=Jn_\Sigma^H,\\
    \label{E2}  E_2&=&\frac{l}{l_W}(\overline{r}\,\overline{p}\,e_1+\overline{r}\,\overline{q}\,e_2-T),
\end{eqnarray}
are both in the kernel of $du$; therefore, they span ${\rm T}\Sigma$ at each non-characteristic point.

For the following proposition, whose proof is strictly computational, we need some notation first:
\begin{eqnarray*}
    \fH&=&e_1(\overline{p})+e_2(\overline{q})+a_3\overline{q}-b_3\overline{p},\\
    \calQ_1(\overline{p},\overline{q})&=&(a_2+b_1)(\overline{q}^2-\overline{p}^2)+2(a_1-b_2)\overline{pq},\\
    \calQ_2(\overline{p},\overline{q})&=&b_2\overline{p}^2+a_1\overline{q}^2-(a_2+b_1)\overline{pq}.
\end{eqnarray*}
\begin{prop}\label{prop-da}
Let $\alpha_1,\alpha_2, \alpha_\Sigma$ be the coframe of $E_1,E_2,n_\Sigma$. Then the following hold:
 \begin{align*}
   d\alpha_1&=A_1\,\alpha_\Sigma\wedge\alpha_1+A_2\,\alpha_\Sigma\wedge\alpha_2+A_3\,\alpha_1\wedge\alpha_2,\\
    d\alpha_2&=B_1\,\alpha_\Sigma\wedge\alpha_1+B_2\,\alpha_\Sigma\wedge\alpha_2+B_3\,\alpha_1\wedge\alpha_2,\\
    d\alpha_\Sigma&=C_1\,\alpha_\Sigma\wedge\alpha_1+C_2\,\alpha_\Sigma\wedge\alpha_2,
\end{align*}
where \begin{align*}
    A_1&=\frac{\fH+ \overline{r}\calQ_2(\overline{p},\overline{q})}{\sqrt{1+\overline{r}^2}},\\
    A_2&=-E_1(\overline{r})-\overline{r}E_1(\log l)-\calQ_1(\overline{p},\overline{q}),\\
    A_3&=\frac{-\overline{r}\fH+\calQ_2(\overline{p},\overline{q})}{\sqrt{1+\overline{r}^2}},\\
     B_1&=\frac{\overline{r}^2-1}{\overline{r}^2+1}E_1(\overline{r})+1,\\
    B_2&=\frac{\overline{r}}{\overline{r}^2+1}n_\Sigma(\overline{r})-E_2(\overline{r}),\\
    B_3&=\frac{\overline{r}}{1+\overline{r}^2}E_1(\overline{r})+\overline{r},\\
    C_1&=\frac{\overline{r}}{\overline{r}^2+1}E_1(\overline{r})+E_1(\log l),\\
    C_2&=\frac{\overline{r}}{\overline{r}^2+1}E_2(\overline{r})+E_2(\log l).
\end{align*}
For a detailed derivation of proposition \ref{prop-da} see \cite{BPPT}.
\end{prop}
\begin{cor}\label{cor-brackets}
Let $A_i,B_i$, $i=1,2,3$ and $C_i$, $i=1,2$ be as above. The following hold:
\begin{align*}
[E_1,E_2]&= -A_3 E_1-B_3 E_2,\\
[E_1,n_\Sigma]&=A_1E_1+B_1E_2+C_1 n_\Sigma,\\
[E_2,n_\Sigma]&=A_2E_1+B_2E_2+C_2 n_\Sigma.
\end{align*}
\end{cor}
\begin{proof}
We prove only the first relation; the remaining ones may be proved in a similar manner. In the first place,
$$
[E_1,E_2]=\alpha_1([E_1,E_2])\,E_1+\alpha_2([E_1,E_2])\,E_2+\alpha_\Sigma([E_1,E_2])\,n_\Sigma.
$$
We conclude the proof by observing that \begin{align*}
\alpha_1([E_1,E_2])&=-d\alpha_1(E_1,E_2)=-A_3,\\
\alpha_2([E_1,E_2])&=-d\alpha_2(E_1,E_2)=-B_3,\\
\alpha_\Sigma([E_1,E_2])&=-d\alpha_\Sigma(E_1,E_2)=0.
\end{align*}
\end{proof}

\subsection{Structural equations}\label{sec-structural}
Consider the following quantities: \begin{align*} \beta_{11}&=-\frac{\fH}{\sqrt{1+\overline{r}^2}},\\
\beta_{12}&=\frac{E_1(\overline{r})}{1+\overline{r}^2},\\
\beta_{21}&=\frac{\overline{r}^2+E_1(\overline{r})+(1/2)\mathcal{Q}_1(\overline{p},\overline{q})}{1+\overline{r}^2},\\
\beta_{22}&=\frac{E_2(\overline{r})}{1+\overline{r}^2}.
\end{align*}
The first structural equations for the moving frame are given by
$$
\left(\begin{matrix}
d\alpha_1\\
d\alpha_2\\
d\alpha_\Sigma
\end{matrix}\right)+
\left(\begin{matrix}
0&\eta_{12}&\eta_{13}\\
-\eta_{12}&0&\eta_{23}\\
-\eta_{13}&-\eta_{23}&0
\end{matrix}\right)\wedge
\left(\begin{matrix}
\alpha_1\\
\alpha_2\\
\alpha_\Sigma
\end{matrix}\right)=
\left(\begin{matrix}
\Theta_1\\
\Theta_2\\
\Theta_3
\end{matrix}\right),
$$
where the connection forms $\eta_{ij}$ $i,j \in \{1,2,3\}$ are given by
\begin{align*}
\eta_{12}&=\overline{r}\beta_{11}\alpha_1+\overline{r}\beta_{21}\alpha_2,\\
\eta_{13}&=\beta_{11}\alpha_1+\beta_{12}\alpha_2,\\
\eta_{23}&=\beta_{21}\alpha_1+\beta_{22}\alpha_2
\end{align*}
and $\Theta_i$, for $i \in \{1,2,3\}$ are given by \begin{align*}
    \Theta_1&=(A_3-\overline{r}\beta_{11})\alpha_1\wedge \alpha_2 \\
     \Theta_2&=(B_3-\overline{r}\beta_{21})\alpha_1\wedge \alpha_2\\
     \Theta_3 &=(\beta_{12}-\beta_{21})\alpha_1\wedge \alpha_2.
\end{align*}
In particular,
$$
\Omega_{12}=d\eta_{12}-\eta_{13}\wedge\eta_{23},
$$
and the {\it sectional curvature} $K(E_1,E_2)=\Omega_{12}$ of the distinguished plane spanned by $E_1$ and $E_2$ is \begin{eqnarray}\label{sectional} \notag K(E_1,E_2)&=&d\eta_{12}(E_1,E_2)- \eta_{13}\wedge\eta_{23}(E_1,E_2)\\ \notag&=&E_1(\eta_{12}(E_2))-E_2(\eta_{12}(E_1))-\eta_{12}([E_1,E_2])\\ \notag&&-\eta_{13}(E_1)\eta_{23}(E_2)+\eta_{13}(E_2)\eta_{23}(E_1)\\ &=&\beta_{12} \beta_{21}-\beta_{22}\beta_{11}-E_1(\overline{r} \beta_{21})+E_2(\overline{r}\beta_{11})-A_3\overline{r}\beta_{11}-B_3\overline{r} \beta_{21}.
\end{eqnarray}
\subsection{Second fundamental form-mean curvature}\label{sec-sec-form}
For a general discussion about the second  fundamental form of CR submanifolds, see~\cite{Dragomir1995}. Denote by $\nabla^\Sigma$ the restriction of the Tanaka-Webster connection $\nabla$ to ${\rm T}\Sigma$; then
$$
\nabla_XY=\nabla^\Sigma_XY+B(X,Y)n_\Sigma,
$$
for every $X,Y\in{\rm T}\Sigma$. Here, $B$ is the {\it Webster second fundamental form} of $\Sigma$. It is now evident that 
\begin{align*}
\nabla^\Sigma_{E_1}E_1&=\overline{r}\beta_{11}\,E_2,&\nabla^\Sigma_{E_2}E_1&=\overline{r}\beta_{21}\,E_2,\\
\nabla^\Sigma_{E_1}E_2&=-\overline{r}\beta_{11}\,E_1,&\nabla^\Sigma_{E_2}E_2&=\overline{r}\beta_{21}\,E_1
\end{align*}
and the matrix of the second fundamental form is
$$
B=\left(\begin{matrix}
\beta_{11}&\beta_{12}\\
\beta_{21}&\beta_{22}
\end{matrix}\right).
$$
The matrix $B$ is not symmetric. In fact, from the expressions for $\beta_{ij}$ we have:
\begin{equation*}
\beta_{21}-\beta_{12}=\frac{\overline{r}^2+(1/2)\calQ_1(\overline{p},\overline{q})}{1+\overline{r}^2}.
\end{equation*}
Geometrically, this lack of symmetry is a direct consequence of the intrinsic torsion of the Tanaka-Webster connection. In classical Riemannian geometry, a torsion-free Levi-Civita connection inherently yields a symmetric second fundamental form. In the present contact setting, however, the Lie bracket of the horizontal frame fields $[E_1, E_2]$ possesses a non-vanishing component along the Reeb vector field $T$. The skew-symmetric part of $B$, quantified by the difference $\beta_{21}-\beta_{12}$, captures the interplay between the surface normal and this ambient torsion. Specifically, it measures how the ambient contact distribution ``twists'' relative to the surface, reflecting the necessary projection of the Reeb vector field onto the normal bundle of $\Sigma$ at non-characteristic points.

\begin{defn}

    The {\it TW mean curvature} $H^{TW}_{\Sigma}$ of $\Sigma$ is defined as
    \begin{equation}\label{eq:TW-mC}
        H_\Sigma^{TW}=-\frac{{\rm tr}(B)}{2}=-\frac{\beta_{11}+\beta_{22}}{2}=\frac{\fH\sqrt{1+\overline{r}^2}-E_2(\overline{r})}{2(1+\overline{r}^2)}.
\end{equation}
    \begin{prop}\label{prop-HTW}
        The TW mean curvature $H_\Sigma^{TW}$ of $\Sigma$ depends only on the constants of $G$ and the horizontal derivatives of the defining function of  $\Sigma$.
\end{prop}
\begin{proof}
    Since
    $$
    \fH=e_1(\overline{p})+e_2(\overline{q})+a_3\overline{q}-b_3\overline{p},
    $$
    we only have to show that $E_2(\overline{r})$  depends only on the horizontal geometry of $\Sigma$. This follows from the relations
    \begin{equation}\label{eq-r-e}
    \overline{r}=\frac{Tu}{l}=a_3\overline{p}+b_3\overline{q}-\frac{[e_1,e_2](u)}{l},
    \end{equation}
    and also
    \begin{eqnarray*}
     E_2&=&(l/l_W)\left(\overline{rp}e_1+\overline{rq}e_2-T\right) \\
     &=&\frac{r}{\sqrt{1+\overline{r}^2}}\,J E_1-\frac{1}{\sqrt{1+\overline{r}^2}}\left([e_1,e_2]-a_3\,e_1+b_3\,e_2\right).
\end{eqnarray*}
\end{proof}
\end{defn}
\subsection{Gauss formula-Gauss curvature}\label{sec-Gauss}
We have seen that the curvature form $\Omega_{12}$ satisfies
$$
\Omega_{12}=d\eta_{12}-\eta_{13}\wedge\eta_{23}.
$$
When restricted to $\Sigma$, $\Omega_{12}^\Sigma=d\eta_{12}$ and therefore
$$
\Omega_{12}=\Omega_{12}^\Sigma-\eta_{13}\wedge\eta_{23}.
$$
If $K_\Sigma^{TW}(E_1,E_2)$ is the sectional curvature of the distinguished plane spanned by $E_1$ and $E_2$ for the restricted metric in $\Sigma$ we therefore have the {\it Gauss formula}
\begin{equation}\label{Gauss-formula}
 K_\Sigma^{TW}(E_1,E_2)=K(E_1,E_2)+\det(B).
\end{equation}
Combining \eqref{Gauss-formula} with \eqref{sectional}
we state the following:
\begin{defn}\label{defn:Gauss-curv}
 The Gauss curvature $K_\Sigma^{TW}$ for the hypersurface $\Sigma$ endowed with the restriction to $\Sigma$ of the Tanaka-Webster metric in $G$ is given by
 \begin{equation}\label{Gauss-curvature}
     K_\Sigma^{TW}=-E_1(\overline{r} \beta_{21})+E_2(\overline{r}\beta_{11})-A_3\overline{r}\beta_{11}-B_3\overline{r} \beta_{21}.
\end{equation}
\end{defn}
Using the relations:
\begin{eqnarray*}
    \overline{r}\beta_{11}&=&-\frac{\overline{r}\fH}{\sqrt{1+\overline{r}^2}},\\
    \overline{r}\beta_{21}&=&\frac{\overline{r}}{1+\overline{r}^2}\left(E_1(\overline{r})+\overline{r}^2+(1/2)\calQ_1(\overline{p},\overline{q})\right),\\
    A_3&=&\frac{-\overline{r}\fH+\calQ_2(\overline{p},\overline{q})}{1+\overline{r}^2},\\
    B_3&=&\frac{\overline{r}}{1+\overline{r}^2}E_1(\overline{r})+\overline{r},
\end{eqnarray*}
we may express $K_\Sigma^{TW}$ in \eqref{Gauss-curvature} using only the constants of the group and derivatives of the defining function $u$ of $\Sigma$.
\begin{rem}\label{rem-cyl}
    A hypersurface $\Sigma$ is called {\it TW--cylindrical} if $Tu\equiv 0$. It is evident that all {\it TW}--cylindrical surfaces $\Sigma$ have $K_\Sigma^{TW}\equiv 0$.
\end{rem}

The following proposition is the analogue of Theorema Egregium and its proof follows from the proof of Proposition \eqref{prop-HTW} and the above relations for $\overline{r}\beta_{11}$, $\overline{r}\beta_{21}$, $A_3$ and $B_3$:
\begin{prop}\label{prop-TW-egregium}
  The Gauss curvature $K_\Sigma^{TW}$ depends only on the constants of $G$ and the horizontal derivatives of the defining function $u$ of $\Sigma$.
\end{prop}

\section{Computations in the Heisenberg Group}\label{sec-comp-H}
In this section, we apply the moving frame machinery established in Section \ref{sec-TW-hyp} to the 3-dimensional Heisenberg group $\mathbb{H}$. We provide  closed-form computations of the Tanaka-Webster mean and Gauss curvatures for several prototypical classes of hypersurfaces, i.e., tilted planes, graphs and surfaces of revolution; this demonstrates how the geometry behaves away from the characteristic locus.

The Heisenberg group $\mathbb{H}$ is $\R^2\times\R$ with coordinates $(x,y,t)$ and group operation
$$
p\cdot p'=(x+x',\,y+y',\,t+t'+2(xy'-yx')),
$$
for every $p=(x,y,t)$ and $p'=(x',y',t')\in\mathbb{H}$. The contact form is
    $$
    \theta=dt+2x\,dy-2y\,dx,
    $$
    the Haar volume form is
    $$
    \omega\wedge d\omega=4\,dx\wedge dy\wedge dt,
    $$
    and the normalised left-invariant frame vector fields are
    $$
    e_1=(1/2)\partial_x+y\partial_t,\quad e_2=(1/2)\partial_y-x\partial_t,\quad T=\partial_t.
    $$
    The only non-zero Lie bracket is $[e_1,e_2]=-T$; therefore $a_i=b_i=0$, $i=1,2,3,$ and $\mathbb{H}$ 
has identically zero Webster scalar curvature. 

Let $\Sigma \subset \mathbb{H}$ be a smooth hypersurface defined as the non-degenerate level set $u(x,y,t) = 0$. At non-characteristic points,  the horizontal gradient norm and the normalised frame components are:
$$
l = \sqrt{(e_1u)^2 + (e_2u)^2}, \quad \overline{p} = \frac{e_1u}{l}, \quad \overline{q} = \frac{e_2u}{l}, \quad \overline{r} = \frac{Tu}{l}.
$$
The tangent space of the surface is spanned by the adapted frame fields $E_1 = -\overline{q}e_1 + \overline{p}e_2$ and $E_2 = \frac{1}{\sqrt{1+\overline{r}^2}}(\overline{r}\overline{p}e_1 + \overline{r}\overline{q}e_2 - T)$.
The horizontal mean curvature term reduces to $\fH = e_1(\overline{p}) + e_2(\overline{q})$.
Since the background torsion and structural functions $\mathcal{Q}_i$ vanish, the structural functions of the orthogonal coframe compress to:
$$
\beta_{11} = -\frac{\fH}{\sqrt{1+\overline{r}^2}}, \quad \beta_{21} = \frac{\overline{r}^2 + E_1(\overline{r})}{1+\overline{r}^2},
$$
$$
A_3 = -\frac{\overline{r}\fH}{1+\overline{r}^2}, \quad B_3 = \frac{\overline{r}E_1(\overline{r})}{1+\overline{r}^2} + \overline{r}.
$$
The Tanaka-Webster mean curvature $H_\Sigma^{TW}$ and Gauss curvature $K_\Sigma^{TW}$ are given explicitly by:
\[
H_\Sigma^{TW} = \frac{\fH\sqrt{1+\overline{r}^2} - E_2(\overline{r})}{2(1+\overline{r}^2)},
\]
$$
K_\Sigma^{TW} = -E_1(\overline{r} \beta_{21}) + E_2(\overline{r}\beta_{11}) - A_3\overline{r}\beta_{11} - B_3\overline{r} \beta_{21}.
$$
\begin{ex}
   {\it TW--Cylindrical Surfaces.}
For the TW--cylindrical surfaces we have:
$$
H_\Sigma^{TW} = \frac{\fH}{2}, \quad K_\Sigma^{TW} = 0.
$$
Here TW--cylindrical surfaces are defined by equations of the form $u(x,y)=0$.
\end{ex}
\begin{ex}{\it Tilted Planes.}
Consider the  tilted plane defined by $u(x,y,t) = ax + by + ct + d = 0$, where $c \neq 0$.
The frame derivatives are:
$$
e_1u = \frac{1}{2}a + cy, \quad e_2u = \frac{1}{2}b - cx, \quad Tu = c.
$$
The horizontal gradient is $l = \sqrt{(\frac{1}{2}a + cy)^2 + (\frac{1}{2}b - cx)^2}$ and is non-zero away from the single characteristic point $C=(b/(2c),\,-a/(2c),-d/c)$.
Evaluating the derivatives of the normal fields yields:
$$
 \fH = e_1(\overline{p}) + e_2(\overline{q}) = 0.
$$
Since $\fH = 0$, it follows that $\beta_{11} = 0$ and $A_3 = 0$.
Next, we find the directional derivatives of $\overline{r} = c/l$:
$$
e_1(\overline{r}) = \frac{1}{2}\overline{r}^2\overline{q}, \quad e_2(\overline{r}) = -\frac{1}{2}\overline{r}^2\overline{p} \implies E_1(\overline{r}) = -\frac{1}{2}\overline{r}^2, \quad E_2(\overline{r}) = 0.
$$
Substituting these into the curvature formulae yields:
$$
H_\Sigma^{TW} = \frac{0 \cdot \sqrt{1+\overline{r}^2} - 0}{2(1+\overline{r}^2)} = 0,
$$
$$
\beta_{21} = \frac{\overline{r}^2 - \frac{1}{2}\overline{r}^2}{1+\overline{r}^2} = \frac{\overline{r}^2}{2(1+\overline{r}^2)}, \quad B_3 = \overline{r}\left(1 - \frac{\overline{r}^2}{2(1+\overline{r}^2)}\right) = \overline{r}\frac{2+\overline{r}^2}{2(1+\overline{r}^2)}.
$$
The Gauss curvature relation simplifies via the one-dimensional chain rule with respect to $\overline{r}$:
\begin{align*}
K_\Sigma^{TW} &= -E_1\left(\frac{\overline{r}^3}{2(1+\overline{r}^2)}\right) - B_3\left(\frac{\overline{r}^3}{2(1+\overline{r}^2)}\right) \\
&= -\left[\frac{\overline{r}^2(3+\overline{r}^2)}{2(1+\overline{r}^2)^2}\right]\left(-\frac{1}{2}\overline{r}^2\right) - \left[\frac{\overline{r}(2+\overline{r}^2)}{2(1+\overline{r}^2)}\right]\left[\frac{\overline{r}^3}{2(1+\overline{r}^2)}\right] \\
&= \frac{\overline{r}^4(3+\overline{r}^2) - \overline{r}^4(2+\overline{r}^2)}{4(1+\overline{r}^2)^2} = \frac{\overline{r}^4}{4(1+\overline{r}^2)^2}.
\end{align*}
Hence, the tilted plane is a TW-minimal surface ($H_\Sigma^{TW} = 0$) with a strictly positive non-vanishing Gauss curvature 
$$
K_\Sigma^{TW} = \frac{\overline{r}^4}{4(1+\overline{r}^2)^2}, \quad \overline{r}=\frac{c}{\sqrt{(\frac{1}{2}a + cy)^2 + (\frac{1}{2}b - cx)^2}},
$$
away from the characteristic point $C$.
\end{ex} 
\subsection{Graphs $t = f(x,y)$}
Let $\Sigma \subset \mathbb{H}$ be a surface defined as the graph of a smooth function $t = f(x,y)$.
The defining function is given by:
\begin{equation*}
u(x,y,t) = t - f(x,y) = 0.
\end{equation*}
Evaluating the frame vector fields on the defining function $u$, we obtain the structural components:
\begin{align*}
e_1 u &= y - \frac{1}{2}f_x, \\
e_2 u &= -x - \frac{1}{2}f_y, \\
T u &= 1.
\end{align*}
The horizontal gradient norm $l$ and the Reeb component $\overline{r}$ are given by:
\begin{equation*}
l = \sqrt{\left(y - \frac{1}{2}f_x\right)^2 + \left(-x - \frac{1}{2}f_y\right)^2}, \quad \overline{r} = \frac{1}{l}.
\end{equation*}
To evaluate the curvatures, we compute the second-order directional derivatives of $u$ with respect to the horizontal frame fields:
\begin{align*}
e_1(e_1 u) &= \left(\frac{1}{2}\partial_x + y\partial_t\right)\left(y - \frac{1}{2}f_x\right) = -\frac{1}{4}f_{xx}, \\
e_2(e_1 u) &= \left(\frac{1}{2}\partial_y - x\partial_t\right)\left(y - \frac{1}{2}f_x\right) = \frac{1}{2} - \frac{1}{4}f_{xy}, \\
e_1(e_2 u) &= \left(\frac{1}{2}\partial_x + y\partial_t\right)\left(-x - \frac{1}{2}f_y\right) = -\frac{1}{2} - \frac{1}{4}f_{yx}, \\
e_2(e_2 u) &= \left(\frac{1}{2}\partial_y - x\partial_t\right)\left(-x - \frac{1}{2}f_y\right) = -\frac{1}{4}f_{yy}.
\end{align*}
Differentiating the horizontal gradient norm $l$ gives:
\begin{align*}
e_1(l) &= \frac{1}{l}\left[ (e_1 u)e_1(e_1 u) + (e_2 u)e_1(e_2 u) \right] = \frac{1}{l}\left[ -\frac{1}{4}(e_1 u)f_{xx} - \left(\frac{1}{2} + \frac{1}{4}f_{xy}\right)(e_2 u) \right], \\
e_2(l) &= \frac{1}{l}\left[ (e_1 u)e_2(e_1 u) + (e_2 u)e_2(e_2 u) \right] = \frac{1}{l}\left[ \left(\frac{1}{2} - \frac{1}{4}f_{xy}\right)(e_1 u) - \frac{1}{4}(e_2 u)f_{yy} \right].
\end{align*}
Thus,
\begin{equation*}
\overline{p}e_1(l) + \overline{q}e_2(l) = -\frac{1}{4l^2}\left[ (e_1 u)^2 f_{xx} + 2(e_1 u)(e_2 u) f_{xy} + (e_2 u)^2 f_{yy} \right].
\end{equation*}
We next proceed to calculate the horizontal mean curvature Field $\mathfrak{H}$:
\begin{equation*}
\mathfrak{H} = e_1(\overline{p}) + e_2(\overline{q}) = \frac{e_1(e_1 u) + e_2(e_2 u)}{l} - \frac{\overline{p}e_1(l) + \overline{q}e_2(l)}{l}.
\end{equation*}
Substituting the expressions derived above, we obtain the following.
\begin{align*}
\mathfrak{H} &= -\frac{1}{4l}(f_{xx} + f_{yy}) + \frac{1}{4l^3}\left[ (e_1 u)^2 f_{xx} + 2(e_1 u)(e_2 u) f_{xy} + (e_2 u)^2 f_{yy} \right] \nonumber \\
&= -\frac{1}{4l^3}\left[ \left(l^2 - (e_1 u)^2\right)f_{xx} - 2(e_1 u)(e_2 u) f_{xy} + \left(l^2 - (e_2 u)^2\right)f_{yy} \right].
\end{align*}
Using the relation $l^2 = (e_1 u)^2 + (e_2 u)^2$, this elegantly reduces to
\begin{equation}\label{hmc-gr-h}
\mathfrak{H} = -\frac{1}{4l^3}\left[ (e_2 u)^2 f_{xx} - 2(e_1 u)(e_2 u) f_{xy} + (e_1 u)^2 f_{yy} \right].
\end{equation}
Now, differentiating $\overline{r} = 1/l$ along $E_2$ yields:
\begin{equation*}
E_2(\overline{r}) = -\frac{1}{l^2}E_2(l) = -\frac{1}{l^2\sqrt{l^2+1}}\left( \overline{p}e_1(l) + \overline{q}e_2(l) \right).
\end{equation*}
Therefore, combining our expressions for $\mathfrak{H}$ and $\overline{p}e_1(l) + \overline{q}e_2(l)$ , we end up with  $\Sigma$:
\begin{equation}\label{mc-g-h}
H_\Sigma^{TW} = \frac{-\left(e_1(u)^2
   \left(l^2+1\right)+e_2(u)^2\right)
   f_{yy}-\left(e_1(u)^2+e_2(u)^2
   \left(l^2+1\right)\right) f_{xx}+2 e_1(u)
   e_2(u) l^2 f_{xy}}{8 l^2
   \left(l^2+1\right)^{3/2}},
\end{equation}
where $e_1 u = y - \frac{1}{2}f_x$ and $e_2 u = -x - \frac{1}{2}f_y$.
The invariant structural relations collapse to:
\begin{equation*}
K_\Sigma^{TW} = -E_1\left(\overline{r}\beta_{21}\right) + E_2\left(\overline{r}\beta_{11}\right) - \frac{\mathfrak{H}^2}{l^2+1} - B_3\overline{r}\beta_{21},
\end{equation*}
where the underlying components are expressed strictly as functions of $l$ and $\mathfrak{H}$ via
\begin{align*}
\overline{r}\beta_{11} &= -\frac{\mathfrak{H}}{\sqrt{l^2+1}}, \\
\overline{r}\beta_{21} &= \frac{1 - E_1(l)}{l(l^2+1)}, \\
B_3 &= \frac{l^2 + 1 - E_1(l)}{l(l^2+1)}.
\end{align*}
The derivative $E_1(l)$ is:
\begin{align*}
E_1(l) &= -\overline{q}e_1(l) + \overline{p}e_2(l) \nonumber \\
&= \frac{1}{2} + \frac{1}{4l^2}\left[ (e_1 u)(e_2 u)(f_{xx} - f_{yy}) - \left((e_1 u)^2 - (e_2 u)^2\right)f_{xy} \right].
\end{align*}
Thus, we may express the Tanaka--Webster Gauss curvature $K_\Sigma^{TW}$ of the graph implicitly by:
\begin{equation}
K_\Sigma^{TW} = -E_1\left( \frac{1 - E_1(l)}{l(l^2+1)} \right) - E_2\left( \frac{\mathfrak{H}}{\sqrt{l^2+1}} \right) - \frac{\mathfrak{H}^2}{l^2+1} - \frac{\left(l^2+1 - E_1(l)\right)\left(1 - E_1(l)\right)}{l^2(l^2+1)^2}.
\end{equation}

\begin{ex}{\it The Saddle Surface.}
Consider the saddle surface defined as the graph $t = xy$. The defining function is $u(x,y,t) = t - xy = 0$, so $f(x,y) = xy$.
We evaluate the first and second partial derivatives:
$$f_x = y, \quad f_y = x, \quad f_{xx} = 0, \quad f_{yy} = 0, \quad f_{xy} = 1.$$
Evaluating the left-invariant horizontal frame fields yields:
\begin{align*}
e_1 u &= y - \frac{1}{2}y = \frac{1}{2}y, \\
e_2 u &= -x - \frac{1}{2}x = -\frac{3}{2}x.
\end{align*}
The squared horizontal gradient norm is:
$$l^2 = (e_1 u)^2 + (e_2 u)^2 = \frac{1}{4}y^2 + \frac{9}{4}x^2.$$
Since $f_{xx}$ and $f_{yy}$ vanish, the horizontal mean curvature field simplifies significantly:
$$\fH = -\frac{1}{4l^3}\left[ - 2\left(\frac{1}{2}y\right)\left(-\frac{3}{2}x\right) \right] = -\frac{3xy}{8l^3}.$$
Applying the graph formula for the Tanaka--Webster mean curvature:
$$H_\Sigma^{TW} = -\frac{- 2(e_1 u)(e_2 u)f_{xy}}{8(l^2+1)^{3/2}} = -\frac{3xy}{16(l^2+1)^{3/2}}.$$
This demonstrates that the saddle surface has a TW mean curvature that vanishes precisely along the coordinate axes $x=0$ and $y=0$, but is non-zero elsewhere. The Gauss curvature $K_\Sigma^{TW}$ can be computed by substituting $l$ and $\fH$ into the structural formulas provided previously.
\end{ex}

\subsection{Surfaces of Revolution in the Heisenberg Group $\mH$}
 Let $\Sigma \subset \mH$ be a surface of revolution generated by a smooth profile curve $t = g(\rho)$ in cylindrical coordinates, where $\rho = \sqrt{x^2+y^2}$;
the defining function is given by $u(x,y,t) = t - g(\rho) = 0$ \cite{BPPT}.
Evaluating the left-invariant frame $e_1 = \frac{1}{2}\partial_x + y\partial_t$, $e_2 = \frac{1}{2}\partial_y - x\partial_t$, and $T = \partial_t$ on $u$ yields:
\begin{align*}
e_1 u &= -\frac{1}{2}g'(\rho)\cos\phi + \rho\sin\phi, \\
e_2 u &= -\frac{1}{2}g'(\rho)\sin\phi - \rho\cos\phi, \\
Tu &= 1.
\end{align*} 
The horizontal gradient norm $l$ and the Reeb component $\overline{r}$ are given by:
\begin{equation}
l = \sqrt{\frac{1}{4}g'(\rho)^2 + \rho^2}, \quad \overline{r} = \frac{1}{l}.
\end{equation}
The horizontal mean curvature $\fH$ reduces to:
\begin{equation}\label{eq:H_field_H}
\fH = e_1(\overline{p}) + e_2(\overline{q}) =\frac{\rho \left(f'(\rho)-\rho
   f''(\rho)\right)}{4l^3} .
\end{equation}
The frame directional derivatives acting on functions of $\rho$ are:
\begin{equation*}
E_1(F(\rho)) = \frac{\rho}{2l}F'(\rho), \quad E_2(F(\rho)) = -\frac{g'(\rho)}{4l\sqrt{l^2+1}}F'(\rho).
\end{equation*}
Applying this to the definition of the TW mean curvature $H_\Sigma^{TW}$, we find:
\begin{equation}\label{eq:H-sor-H}
H_\Sigma^{TW} = \frac{\fH\sqrt{1+\overline{r}^2}-E_2(\overline{r})}{2(1+\overline{r}^2)} =\frac{-\left(l^2+1\right) \rho^2
   g''(\rho)+l^3 g'(\rho)
   \overline{r}'(\rho)+\left(l^2+1\right) \rho
   g'(\rho)}{8 l^2
   \left(l^2+1\right)^{3/2}} .
\end{equation}

The TW Gauss curvature is given by
\begin{equation}\label{eq:K-sor-H}
K_\Sigma^{TW} = -\frac{\rho}{2l} \left[ \frac{2l - \rho l'}{2l^2(l^2+1)} \right]' + \frac{g'(\rho)}{4l\sqrt{l^2+1}} \left[ \frac{\fH}{\sqrt{l^2+1}} \right]' - \frac{\fH^2}{l^2+1} - \frac{(2l^3 + 2l + \rho l^2 r')(2l - \rho l')}{4l^4(l^2+1)^2},
\end{equation}
where $l' = \frac{g'(\rho)g''(\rho) + 4\rho}{4l}$.
\subsubsection*{Constant Curvature Equations}
The constant TW Gauss curvature relation $K_\Sigma^{TW} = c_1$ is a fourth-order nonlinear ordinary differential equation for $g(\rho)$;
on the other hand, we may
list some solutions of the constant TW mean curvature equation $K_\Sigma^{TW} = c_2$ below:
\begin{enumerate}
    \item Horizontal Planes ($c_1 = 0$): setting $g'(\rho) = 0 \implies g(\rho) = t_0$ solves the ODE identically.
Thus, horizontal planes are TW-minimal surfaces.
    \item Vertical Cylinders: for a profile parameterised by $\rho = R$, the horizontal gradient becomes $l = 1/2$, yielding the constant curvature solution $H_\Sigma^{TW} = \frac{1}{4R}$.
\end{enumerate}

\begin{ex}{\it The CC-Sphere (Pansu Sphere).}
The Carnot-Carath\'{e}odory (CC) sphere, or Pansu sphere, centred at $0$ and of radius $R$, is given by the surface patch
        $$
        \sigma(k, \phi) =
\left(\frac{1-e^{ikR}}{k}e^{i\phi},\, \frac{2}{k^2}(\sin(k R)- kR)\right),
$$
with $(k,\phi)\in(-2\pi/R,2\pi/R)\times(0,2\pi)$. We write
$$
r=A(k)=\frac{2\sin(kR/2)}{k},\quad t=B(k)=\frac{2}{k^2}(\sin(kR)-kR).
$$
Then $t=g(r)=(B\circ A^{-1})(r)$ and we calculate
\begin{eqnarray*}
 g'(r)&=&\frac{4}{k}\cos (kR/2),
\end{eqnarray*}
and $(g')^2+4r^2=16/k^2$. 
Moreover,
$$
 g''(r)=-2\frac{\cos(kR/2)+2(kR)\sin(kR/2)}{(kR)\cos(kR/2)-2\sin(kR/2)}.
$$
The TW Gauss and mean curvature may now be calculated using the above formulae for the derivatives of $g$; the results though are quite long formulae that we shall not include here. We only stress that none of the $H^{TW}$ and $K^{TW}$ are constant.
\end{ex}

\begin{ex}
    The {\it bubble set $\mathcal{B}_R=\mathcal{B}(O,R)$ centred at $O$ and of radius $R>0$}  is the surface of revolution defined by the formula
        $$
        t=f(r)=r\sqrt{4R^2-r^2}-4R^2\arcsin(r/(2R))+2\pi R^2.
        $$
        In contrast with the horizontal mean curvature (see \cite{BPPT}), the bubble set has non-constant TW mean curvature.
\end{ex}

\section{Computations in the Affine Additive Group}\label{sec-comp-aa}

This section expands our geometric investigation to hypersurfaces embedded within the non-unimodular affine additive Lie group $\Aa$. By looking at specific coordinate level sets and graphs, we explore the  interaction between the group's underlying algebraic structure, its negative background Webster scalar curvature, and the resulting surface invariants.

The affine-additive group $\Aa$, is defined as $\mathcal{AA}:=\mathbb{R} \times \mathbf{H}^1_{\mathbb{C}}$, where $\mathbf{H}^1_\C:=\{(\lambda,t):\lambda>0,t\in\R\}$ is the right half-plane model for the hyperbolic plane. The group operation is given by 
 $$
 p \cdot p'=(a+a',\lambda \lambda' , \lambda t'+t),
 $$ 
 for every $p=(a,\lambda,t)$ and $p'=(a',\lambda',t')$.
The standard adapted left-invariant orthonormal frame is:
\[
e_1 = \lambda\partial_\lambda, \quad e_2 = \partial_a + \lambda\partial_t, \quad T = \partial_a.
\]
The Lie bracket configurations are given by $[e_1, e_2] = e_2 - T$ and $[e_1, T] = [e_2, T] = 0$.
The structural constants for $\mathcal{AA}$ are:
$$
a_1 = b_1 = a_2 = b_2 = a_3 = 0, \quad b_3 = 1.
$$
Consequently, the Webster scalar curvature is $R=-1$.

Let $\Sigma$ be a hypersurface given as usual by  $u=0$;
then, the auxiliary quadratic invariants vanish identically: $\mathcal{Q}_1(\overline{p}, \overline{q}) = \mathcal{Q}_2(\overline{p}, \overline{q})  = 0$.
The horizontal mean curvature reduces to:
$$
\fH = e_1(\overline{p}) + e_2(\overline{q}) - \overline{p}.
$$
Using the specific structural constants of $\mathcal{AA}$, the orthonormal frame coefficients become:
$$
\beta_{11} = -\frac{\fH}{\sqrt{1+\overline{r}^2}}, \quad \beta_{21} = \frac{\overline{r}^2 + E_1(\overline{r})}{1+\overline{r}^2},
$$
$$
A_3 = -\frac{\overline{r}\fH}{\sqrt{1+\overline{r}^2}} = \overline{r}\beta_{11}, \quad B_3 = \frac{\overline{r}E_1(\overline{r})}{1+\overline{r}^2} + \overline{r}.
$$
The Tanaka-Webster mean curvature $H_\Sigma^{TW}$ and Gauss curvature $K_\Sigma^{TW}$ along non-characteristic points of $\Sigma$ are defined by:
$$
H_\Sigma^{TW} = \frac{\fH\sqrt{1+\overline{r}^2} - E_2(\overline{r})}{2(1+\overline{r}^2)},
$$
\[
K_\Sigma^{TW} = -E_1(\overline{r} \beta_{21}) + E_2(\overline{r}\beta_{11}) - A_3\overline{r}\beta_{11} - B_3\overline{r} \beta_{21}
\]

\begin{ex}{\it TW-Cylindrical Surfaces.}
 The TW-cylindrical condition $Tu \equiv 0$ yields the following:  
 $$
\beta_{11} = -\fH, \quad \beta_{21} = 0, \quad A_3 = 0, \quad B_3 = 0.
$$
Therefore,
$$
H_\Sigma^{TW} = \frac{\fH}{2}, \quad K_\Sigma^{TW} = 0.
$$
\end{ex}
\begin{ex}{\it The Plane $a = 0$.}
 Consider the coordinate plane defined by the level set function $u(\lambda, a, t) = a = 0$;
this realises the natural embedding of the hyperbolic plane $\bH^1_\C$ into $\Aa$.
Applying the frame derivatives to $u$ yields:
$$
e_1u = \lambda\partial_\lambda(a) = 0, \quad e_2u = (\partial_a + \lambda\partial_t)(a) = 1, \quad Tu = \partial_a(a) = 1.
$$ 
The horizontal gradient is constant: $l = \sqrt{(e_1u)^2 + (e_2u)^2} = 1$.
The normalised component fields are thus 
$$
\overline{p} = 0, \quad \overline{q} = 1, \quad \overline{r} = 1.
$$
Since $\overline{p}$, $\overline{q}$, and $\overline{r}$ are constant functions across the entire plane, all derivatives vanish identically:
$$
e_1(\overline{p}) = e_2(\overline{q}) = 0 \implies \fH = 0 + 0 - 0 = 0,
$$
$$
E_1(\overline{r}) = E_2(\overline{r}) = 0.
$$
Substituting these vanishing derivatives into the structural coefficients gives:
$$
\beta_{11} = 0, \quad \beta_{21}  = \frac{1}{2},
$$
$$
A_3 = 0, \quad B_3  = 1.
$$
The Tanaka-Webster curvatures for the plane $a=0$ are therefore
$$
H_\Sigma^{TW} = 0,\quad K_\Sigma^{TW}=-\frac{1}{2}.
$$
Hence, the plane $a = 0$ is a TW-minimal surface ($H_\Sigma^{TW} = 0$) embedded in $\mathcal{AA}$ possessing a strictly negative constant Tanaka-Webster Gauss curvature of $K_\Sigma^{TW} = -\frac{1}{2}$.
\end{ex}
\subsection{Graphs $a=f(\lambda, t)$.}
Let $\Sigma$ be a hypersurface defined  as a graph $a = f(\lambda, t)$, so that the defining function is $u(\lambda, a, t) = a - f(\lambda, t) = 0$.
We compute:
\begin{equation*}
e_1(u) = -\lambda f_\lambda, \quad e_2(u) = 1 - \lambda f_t, \quad T(u) = 1.
\end{equation*}
The horizontal gradient norm $l$ and the associated normalised parameters are written as follows:
\begin{equation*}
l = \sqrt{\lambda^2 f_\lambda^2 + (1 - \lambda f_t)^2}, \quad \overline{p} = -\frac{\lambda f_\lambda}{l}, \quad \overline{q} = \frac{1 - \lambda f_t}{l}, \quad \overline{r} = \frac{1}{l}.
\end{equation*}
Since $\overline{p}$ and $\overline{q}$ do not depend on the variable $a$, the horizontal structural function $\fH$ becomes:
\begin{equation*}
\fH = \lambda \frac{\partial \overline{p}}{\partial \lambda} + \lambda \frac{\partial \overline{q}}{\partial t} - \overline{p}.
\end{equation*}
The quadratic forms simplify identically to $\calQ_1(\overline{p},\overline{q}) = \calQ_2(\overline{p},\overline{q})  = 0$.
The adapted intrinsic vector fields spanning $T\Sigma$ act on functions independent of $a$ via:
\begin{align*}
E_1 &= -\lambda \overline{q}\partial_\lambda + \lambda \overline{p}\partial_t, \\
E_2 &= \frac{\lambda \overline{r}}{\sqrt{1+\overline{r}^2}} \left( \overline{p}\partial_\lambda + \overline{q}\partial_t \right).
\end{align*}
In this way, we have:
\begin{align*}
E_1(\overline{r}) &= -\lambda \overline{q} \frac{\partial \overline{r}}{\partial \lambda} + \lambda \overline{p} \frac{\partial \overline{r}}{\partial t}, \\
E_2(\overline{r}) &= \frac{\lambda \overline{r}}{\sqrt{1+\overline{r}^2}} \left( \overline{p}\frac{\partial \overline{r}}{\partial \lambda} + \overline{q}\frac{\partial \overline{r}}{\partial t} \right).
\end{align*}
Applying the structural parameters to the mean curvature definition, we find the explicit formula for $H_\Sigma^{TW}$:
\begin{equation*}
H_\Sigma^{TW} = \frac{\fH\sqrt{1+\overline{r}^2} - E_2(\overline{r})}{2(1+\overline{r}^2)}.
\end{equation*}
Substituting the expressions for $\fH$ and $E_2(\overline{r})$, this yields:
\begin{equation}\label{hmc-aa-g}
H_\Sigma^{TW} = \frac{\left(\lambda \overline{p}_\lambda + \lambda \overline{q}_t - \overline{p}\right)(1+\overline{r}^2) - \lambda \overline{r}\left(\overline{p}\overline{r}_\lambda + \overline{q}\overline{r}_t\right)}{2(1+\overline{r}^2)^{3/2}}.
\end{equation}

Here, the restricted components of the second fundamental form are:
\begin{equation*}
\overline{r}\beta_{11} = -\frac{\overline{r}\fH}{\sqrt{1+\overline{r}^2}}, \quad \overline{r}\beta_{21} = \frac{\overline{r}\left(\overline{r}^2 + E_1(\overline{r})\right)}{1+\overline{r}^2}.
\end{equation*}
Hence, the final expanded representation for the Tanaka--Webster Gauss curvature of the graph is:
\begin{equation}\label{gc-aa-g}
K_\Sigma^{TW} = -E_1\left( \frac{\overline{r}^3 + \overline{r}E_1(\overline{r})}{1+\overline{r}^2} \right) - E_2\left( \frac{\overline{r}\fH}{\sqrt{1+\overline{r}^2}} \right) - \frac{\overline{r}^3 + \overline{r}E_1(\overline{r})}{1+\overline{r}^2}\left(  \frac{\overline{r}}{1+\overline{r}^2}E_1(\overline{r})+\overline{r}
   \right)- \left( \frac{-\overline{r}\fH}{\sqrt{1+\overline{r}^2}}\right)^2.
\end{equation}
\subsection{Surfaces of revolution}
Let $\Sigma \subset \Aa$ be a surface of revolution defined implicitly by the level set function $u(\lambda, a, t) = a - f(\rho) = 0$, where $\rho = t/\lambda$.
The frame derivatives of the defining function $u$ are expressed in terms of the ordinary derivative $f'(\rho)$ as follows:
\begin{equation*}
e_1(u) = \rho f'(\rho), \quad e_2(u) = 1 - f'(\rho), \quad T(u) = 1.
\end{equation*}
The horizontal gradient norm $l$ and the associated normalised parameters depend exclusively on the profile variable $\rho$:
\begin{equation*}
l = \sqrt{\rho^2 (f'(\rho))^2 + (1 - f'(\rho))^2}, \quad \overline{p} = \frac{\rho f'(\rho)}{l}, \quad \overline{q} = \frac{1 - f'(\rho)}{l}, \quad \overline{r} = \frac{1}{l}.
\end{equation*}
The horizontal mean curvature $\fH$ reduces to the one-dimensional profiling relation:
\begin{equation*}
\fH = -\rho \overline{p}'(\rho) + \overline{q}'(\rho) - \overline{p}(\rho).
\end{equation*}
The actions of the frame fields $E_1, E_2 \in T\Sigma$ on any smooth parameter function $h(\rho)$ are given by:
\begin{align}
E_1(h) &= \rho\overline{r} h'(\rho), \\
E_2(h) &= \frac{\overline{r}}{\sqrt{1+\overline{r}^2}} \left( \overline{q}(\rho) - \rho \overline{p}(\rho) \right) h'(\rho).
\end{align}
In particular, the tracking directional derivatives for $\overline{r}(\rho)$ are computed directly by setting $h = \overline{r}$.
Applying the reduced tracking actions to the general formulation, the Tanaka--Webster mean curvature $H_\Sigma^{TW}$ of the surface of revolution satisfies:
\begin{equation*}
H_\Sigma^{TW} = \frac{\fH\sqrt{1+\overline{r}^2} - E_2(\overline{r})}{2(1+\overline{r}^2)},
\end{equation*}
where substituting $E_2(\overline{r})$ yields the explicit ordinary differential form:
\begin{equation}\label{hmc-sor-aa}
H_\Sigma^{TW} = \frac{\fH (1+\overline{r}^2) - \overline{r}(\overline{q} - \rho\overline{p})\overline{r}'(\rho)}{2(1+\overline{r}^2)^{3/2}}.
\end{equation}
Using the constraints $A_3 = -\frac{\overline{r}\fH}{\sqrt{1+\overline{r}^2}}$ and $B_3 = \frac{\overline{r}}{1+\overline{r}^2}E_1(\overline{r}) + \overline{r}$, the Gauss curvature $K_\Sigma^{TW}$ expands to:
\begin{equation*}
K_\Sigma^{TW} = -E_1(\overline{r} \beta_{21}) + E_2(\overline{r}\beta_{11}) - A_3\overline{r}\beta_{11} - B_3\overline{r} \beta_{21},
\end{equation*}
where the second fundamental form components are determined by:
\begin{equation*}
\overline{r}\beta_{11} = -\frac{\overline{r}\fH}{\sqrt{1+\overline{r}^2}}, \quad \overline{r}\beta_{21} = \frac{\overline{r}\left(E_1(\overline{r}) + \overline{r}^2\right)}{1+\overline{r}^2}.
\end{equation*}
Combining these terms yields the final expression for the Tanaka--Webster Gauss curvature of the rotational profile:
\begin{equation}\label{gc-sor-aa}
K_\Sigma^{TW} = -E_1\left( \frac{\overline{r}^3 + \overline{r}E_1(\overline{r})}{1+\overline{r}^2} \right) - E_2\left( \frac{\overline{r}\fH}{\sqrt{1+\overline{r}^2}} \right) - \frac{\overline{r}^2 \fH^2}{1+\overline{r}^2} - \frac{\overline{r}^2\left(1+\overline{r}^2+E_1(\overline{r})\right)\left(\overline{r}^2+E_1(\overline{r})\right)}{(1+\overline{r}^2)^2}.
\end{equation}

\section{Sasakian Structures and Riemannian Approximations}\label{sec-sasaki}
In this final section, we specialise our ambient space to Lie groups carrying a Sasakian structure, see \cite{BG-SG} and analyse how the Tanaka-Webster connection simplifies under these conditions. Furthermore, we investigate the asymptotic behaviour of our curvature profiles under a sequence of Riemannian approximation metrics, reinforcing connections to classical sub-Riemannian limits established in foundational works like \cite{CHMY2005} and \cite{DGN2001}.

For Riemannian metrics in contact manifolds, and in particular, Sasakian structures, see~\cite{Blair2010}. Recall that if $M$ is a $2p+1$-dimensional smooth manifold, then we call $(M,\theta,T, J , g)$ a contact metric structure on $M$ if \begin{enumerate} \item $\theta$ is a smooth 1-form such that $\theta\wedge(d\theta)^p\neq 0$;
\item $T$ is the Reeb field of $\theta$: $\theta(T)=1$, $d\theta(T,\cdot)=0$;
\item in the $2p$-dimensional subbundle $\mathcal{H}=\ker\theta$ there is a complex operator $J:\mathcal{H}\to\mathcal{H}$: $J^2=-{\rm id}_{{\mathcal H}}$ which we can expand to an endomorphism on the whole $TM$ by setting $J T=0$.
\item a canonical Riemannian metric $g$ is defined in $M$ by the relations $$ \theta (U)=g(U,T),\quad d\theta (U,V)=g(J U,V),\quad J^2 (U)=-U+\theta (U)T, $$ for all vector fields $U,V\in \mathfrak{X}(M)$.
\end{enumerate}
In our setting, we immediately have the following: \begin{prop} $(G,\theta,T,J,g)$ where $g$ is the Webster metric is a contact metric manifold.
\end{prop} Let $\{e_1,e_2,T\}$ and $\{\theta_1,\theta_2,\theta\}$ be as in Section \ref{sec:contact}.
The first structural equations for the Levi-Civita connection $\nabla^\theta$ are given by the matrix equation of 1-forms $$ \left(\begin{matrix} d\theta_1\\ d\theta_2\\ d\theta \end{matrix}\right)= \left(\begin{matrix} 0&-\Phi_{12}&-\Phi_{13} \\ \Phi_{12}&0&-\Phi_{23} \\ \Phi_{13}&\Phi_{23}&0 \end{matrix}\right)\wedge\left(\begin{matrix} \theta_1\\ \theta_2\\ \theta \end{matrix}\right).
$$ Straightforward calculations imply the following: \begin{align*} \Phi_{12}&=a_3\,\theta_1+b_3\, \theta_2+\frac{b_1-a_2-1}{2}\, \theta, \\ \Phi_{13}&=a_1\, \theta_1 +\frac{a_2+b_1-1}{2}\,\theta_2, \\ \Phi_{23}&= \frac{b_1+a_2+1}{2}\,\theta_1 +b_2\,\theta_2.
\end{align*} \begin{lem}\label{lem-k-contact} $(G, \theta, T, J, g)$ where $g$ is the Webster metric is $K$-contact if and only if $a_1=b_2=0$ and $a_2+b_1=0$.
\end{lem} \begin{proof} $T$ is Killing if \begin{align*} g(\nabla^\theta_VT,U)+g(\nabla^\theta_UT,V)&=0. \end{align*} Let $U=u_1e_1+u_2e_2+u_3T$ and $V=v_1 e_1+v_2 e_2+v_3 T$, then \begin{align*} \Phi_{13}(V)&=a_1v_1+\frac{a_2+b_1-1}{2}v_2,\\ \Phi_{23}(V)&=\frac{a_2+b_1+1}{2}v_1+b_2v_2.
\end{align*} We have \begin{align*} g_W(\nabla^\theta_VT,U)&=\left( a_1v_1+\frac{a_2+b_1-1}{2}v_2 \right)u_1+\left( \frac{a_2+b_1+1}{2}v_1+b_2v_2 \right)u_2,\\ g_W(\nabla^\theta_UT,V)&=\left( a_1u_1+\frac{a_2+b_1-1}{2}u_2 \right)v_1+\left( \frac{a_2+b_1+1}{2}u_1+b_2u_2 \right)v_2.
\end{align*} Thus \begin{align*} g_W(\nabla^\theta_VT,U)+ g_W(\nabla^\theta_UT,V)&=2a_1u_1v_1+\left( \frac{a_2+b_1-1}{2}+\frac{a_2+b_1+1}{2} \right)v_2u_1 \\ &+\left( \frac{a_2+b_1+1}{2}+\frac{a_2+b_1-1}{2} \right)u_2v_1+2b_2u_2v_2.
\end{align*} Hence, $T$ is Killing if and only if $a_1=b_2=0$ and $a_2+b_1=0$.
\end{proof}
We also considered Riemannian approximations $(G,g_\epsilon)$, $\epsilon>0$ for the sub-Riemannian metric of $G$.
In this context, we calculated the Gauss curvature $K_\Sigma^\epsilon$ of a hypersurface $\Sigma$ isometrically embedded in $G$.
For $\epsilon=1$, $g_1=g$ is the Sasakian metric in $G$ as above.
Explicitly, let $K_\Sigma^1$ be the Levi-Civita Gauss curvature of the surface in $G$.
As seen in \cite{BPPT} , $K^1_\Sigma$ is given by
$$
K^1_\Sigma:-E_1(B_3)+E_2(A_3)-A_3^2-B^2_3.
$$
We have that \begin{align*}
    K_\Sigma^{TW} &= K_\Sigma^1 +E_1\left( \frac{\overline{r}}{1+\overline{r}^2} \right)+B_3 \frac{\overline{r}}{1+\overline{r}^2}.
\end{align*}
For $\epsilon \to 0$, $K^\epsilon_\Sigma$ is given by \begin{align*}
    K^0_\Sigma &=-E_1\left( \frac{Tu}{l} \right) -\left( \frac{Tu}{l} \right)^2,
\end{align*}
hence
\begin{align*}
    K_\Sigma^{TW} &=K^0_\Sigma -E_1\left( \frac{\overline{r}E_1(\overline{r})}{1+\overline{r}^2} \right)+E_2(A_3)-A_3^2-\left( \frac{\overline{r}E_1(\overline{r})}{1+\overline{r}^2} \right)^2 -2\left( \frac{\overline{r}^2E_1(\overline{r})}{1+\overline{r}^2} \right)+E_1\left( \frac{\overline{r}}{1+\overline{r}^2} \right)+B_3\frac{\overline{r}}{1+\overline{r}^2}.
\end{align*}
\begin{rem}
    All TW-cylindrical surfaces have $K_\Sigma^{TW}=K^1_\Sigma =K^0_\Sigma=0.$
\end{rem}

\bibliographystyle{plain}  
\bibliography{Library.bib}    

@misc{BPPT,
 author = {E. Bubani and A. Pinamonti and I. D. Platis and D. Tsolis},
 title = {Horizontal curvatures of surfaces in {3D} contact sub-{Riemannian} {Lie} groups},
 year = {2026},
 howpublished = {Preprint, {arXiv}:2603.08330 [math.{DG}] (2026)},
 url = {https://arxiv.org/abs/2603.08330},
 arXiv = {arXiv:2603.08330}
}

@book{A,
    AUTHOR = {T. Aubin},
     TITLE = {A course in differential geometry},
    SERIES = {Graduate Studies in Mathematics},
    VOLUME = {27},
 PUBLISHER = {American Mathematical Society, Providence, RI},
      YEAR = {2001},
       DOI = {10.1090/gsm/027}
}

@book{BG-SG,
    AUTHOR = {C.P. Boyer and K. Galicki},
     TITLE = {Sasakian geometry},
    SERIES = {Oxford Mathematical Monographs},
 PUBLISHER = {Oxford University Press, Oxford},
      YEAR = {2008}
}

@book{DT-CR,
    AUTHOR = {S. Dragomir and G. Tomassini},
     TITLE = {Differential geometry and analysis on {CR} manifolds},
    SERIES = {Progress in Mathematics},
    VOLUME = {246},
 PUBLISHER = {Birkh\"auser Boston},
      YEAR = {2006},
       DOI = {10.1007/0-8176-4483-0}
}

@article{E-N,
 author = {Eastwood, M. and Neusser, K.},
 title = {A canonical connection on sub-{Riemannian} contact manifolds.},
 fjournal = {Archivum Mathematicum},
 journal = {Arch. Math. (Brno)},
 issn = {0044-8753},
 volume = {52},
 number = {5},
 pages = {277--289},
 year = {2016},
 language = {English},
 doi = {10.5817/AM2016-5-277},
 zbMATH = {6674904},
 Zbl = {1413.53078}
}

@book{Bejancu1986,
  author    = {A. Bejancu},
  title     = {Geometry of {CR}-submanifolds},
  publisher = {Kluwer Academic Publishers},
  year      = {1986},
  address   = {Dordrecht}
}

@book{Blair2010,
  author    = {D. E. Blair},
  title     = {Riemannian Geometry of Contact and Symplectic Manifolds},
  series    = {Progress in Mathematics},
  volume    = {203},
  publisher = {Birkh{\"a}user},
  year      = {2010},
  edition   = {2nd}
}

@article{Dragomir1995,
  author  = {S. Dragomir},
  title   = {On the second fundamental form of a {CR} submanifold},
  journal = {Annali di Matematica Pura ed Applicata},
  volume  = {169},
  number  = {1},
  pages   = {55--83},
  year    = {1995},
  doi     = {10.1007/BF01760634}
}

@article{Lee1988,
  author  = {J. M. Lee},
  title   = {The {Fefferman} metric and pseudo-{Hermitian} invariants},
  journal = {Transactions of the American Mathematical Society},
  volume  = {305},
  number  = {1},
  pages   = {1--22},
  year    = {1988},
  doi     = {10.1090/S0002-9947-1988-0916964-0}
}

@book{Tanaka1975,
  author    = {N. Tanaka},
  title     = {A differential geometric study on strongly pseudo-convex manifolds},
  publisher = {Kinokuniya Book Store},
  year      = {1975},
  address   = {Tokyo}
}

@article{ChernMoser1974,
  author  = {S. S. Chern and J. K. Moser},
  title   = {Real hypersurfaces in complex manifolds},
  journal = {Acta Mathematica},
  volume  = {133},
  number  = {1},
  pages   = {219--271},
  year    = {1974},
  doi     = {10.1007/BF02392144}
}

@article{CHMY2005,
 author = {Hwang, J.-F. and Cheng, J.-H. and Malchiodi, A. and Yang, P.},
 title = {Minimal surfaces in pseudohermitian geometry},
 fjournal = {Annali della Scuola Normale Superiore di Pisa. Classe di Scienze. Serie V},
 journal = {Ann. Sc. Norm. Super. Pisa, Cl. Sci. (5)},
 issn = {0391-173X},
 volume = {4},
 number = {1},
 pages = {129--177},
 year = {2005},
 language = {English},
 url = {https://eudml.org/doc/84552},
 zbMATH = {5058794},
 Zbl = {1158.53306}
}

@article{DGN2001,
 author = {Danielli, D. and Garofalo, N. and Nhieu, D. M.},
 title = {Sub-{Riemannian} calculus on hypersurfaces in {Carnot} groups},
 fjournal = {Advances in Mathematics},
 journal = {Adv. Math.},
 issn = {0001-8708},
 volume = {215},
 number = {1},
 pages = {292--378},
 year = {2007},
 language = {English},
 doi = {10.1016/j.aim.2007.04.004},
 zbMATH = {5196014},
 Zbl = {1129.53017}
}
\Addresses

\end{document}